\documentclass{amsart}
\usepackage{hyperref}
\usepackage{mathrsfs}
\usepackage{amsmath}
\usepackage{amssymb}
\usepackage[all]{xy}
\usepackage{graphicx}
\usepackage{latexsym}
\usepackage[utf8]{inputenc}
\usepackage[T1,T5]{fontenc}
\newtheorem{theorem}{Theorem}
\theoremstyle{definition}
\newtheorem*{acknowledgment}{Acknowledgment}
\begin{document}

\title[On some extensions of Morley's trisector theorem]{On some extensions of Morley's trisector theorem}

\author[Nikos Dergiades]{Nikos Dergiades}
\address{I. Zanna 27, Thessaloniki 54643, Greece}
\email{ndergiades@yahoo.gr}

\author[Hung Quang Tran]{Tran Quang Hung}
\address{High School for Gifted Students, Hanoi University of Science, Vietnam National University, 182 Luong The Vinh Str., Thanh Xuan, Hanoi, Vietnam}
\email{tranquanghung@hus.edu.vn}


\maketitle

\begin{abstract}We establish a simple generalization for the famous theorem of Morley about trisectors in triangle with a purely synthetic proof using only angle chasing and similar triangles. Furthermore, based on the converse construction, another simple extension of Morley's Theorem is created and proven.\end{abstract}

\section{Introduction}Over one hundred years ago in 1899, Frank Morley introduced a geometric result. This result was so classic that Alexander Bogomolny once said "it entered mathematical folklore"; see \cite{1}. Morley's marvelous theorem states as follows:

\begin{theorem}[Morley, 1899]\label{thm1}The three points of intersection of the adjacent trisectors of the angles of any triangle form an equilateral triangle.
\end{theorem}

\begin{figure}[htbp]\begin{center}\scalebox{0.7}{\includegraphics{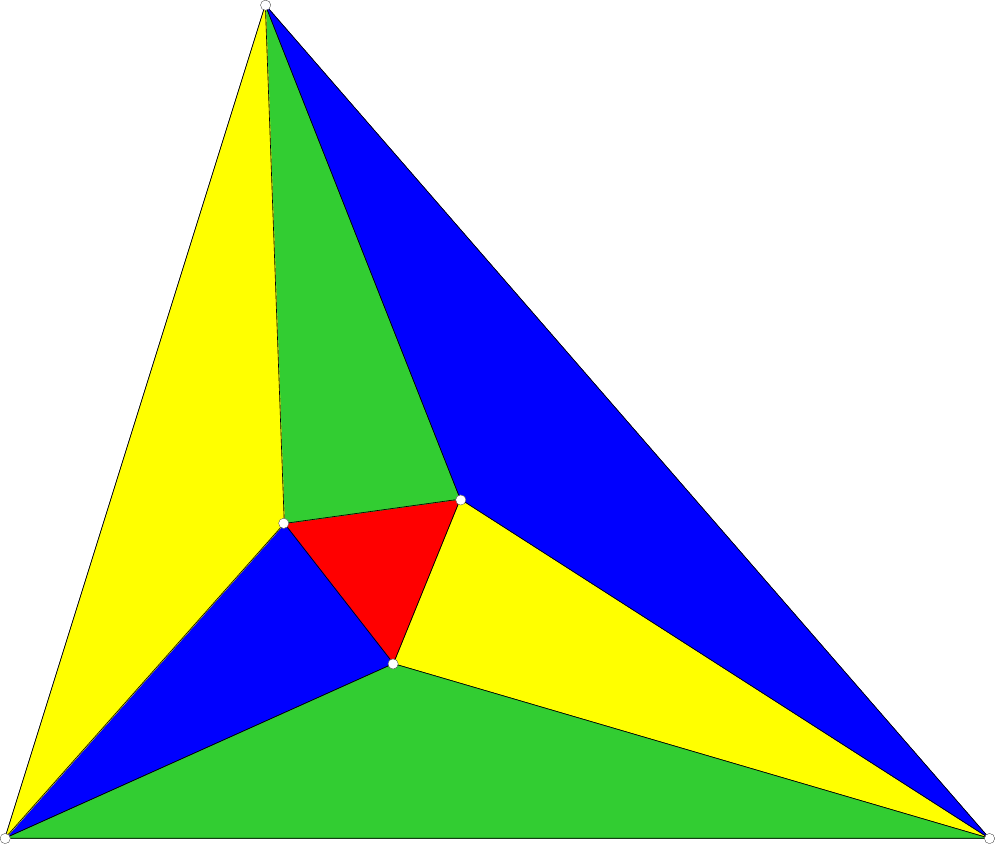}}\end{center}
	\label{fig1}
	\caption{Morley's marvelous theorem}
\end{figure}

Many mathematicians consider Morley's Theorem to be one of the most beautiful theorems in plane Euclidean geometry. Throughout history, numerous proofs have been proposed; see \cite{1,2,6,10,4,7,8}. There was a generalization of Morley's Theorem using projective geometry in \cite{9}. Some extensions to this theorem has been recently analyzed by Richard Kenneth Guy in \cite{3}. Guy's extensions were very extensive and deep research on Morley's Theorem.

In the main part of this paper, we would like to offer and prove synthetically a simple generalization of Morley's Theorem. Generalized theorem is introduced as follows:

\begin{theorem}[Generalized theorem]\label{thm2}Let $ABC$ be a triangle. Assume that three points $X$, $Y$, $Z$, and the intersection $D = BZ \cap CY$, $E = CX \cap AZ$, $F = AY \cap BX$ lie inside triangle $ABC$, they also satisfy the following conditions
	\begin{itemize}
		
		\item $\angle BXC = 120^ \circ + \angle YAZ$, $\angle CYA = 120^ \circ + \angle ZBX$, and $\angle AZB = 120^ \circ + \angle XCY$.
		
		\item The points $X$, $Y$, and $Z$ lie on the bisectors of angles $\angle BDC$, $\angle CEA$, and $\angle AFB$, respectively.
	\end{itemize}	
	Then triangle $XYZ$ is an equilateral triangle.
\end{theorem}

When $X$, $Y$, and $Z$ are the intersections of the adjacent trisectors of triangle $ABC$, it is easily seen that they satisfy two conditions of Theorem \ref{thm2}. Thus Theorem \ref{thm1} is a direct consequence of Theorem \ref{thm2}.

An important property of the pair of triangles $ABC$ and $XYZ$ is introduced in the following theorem:

\begin{theorem}\label{thm3}The triangles $ABC$ and $XYZ$ of the Theorem \ref{thm2} are perspective.
\end{theorem}

At the last section of this paper, we shall apply a converse construction to find another extension for Morley's Theorem. Some new equilateral triangles in a given arbitrary triangle are also found. The family of these new equilateral triangles closely relate to the construction of the Morley's equilateral triangle.

\section{Proofs of the theorems} The solutions for Theorem \ref{thm2} and Theorem \ref{thm3} will be showed in this section.

\begin{figure}[htbp]\begin{center}\scalebox{0.7}{\includegraphics{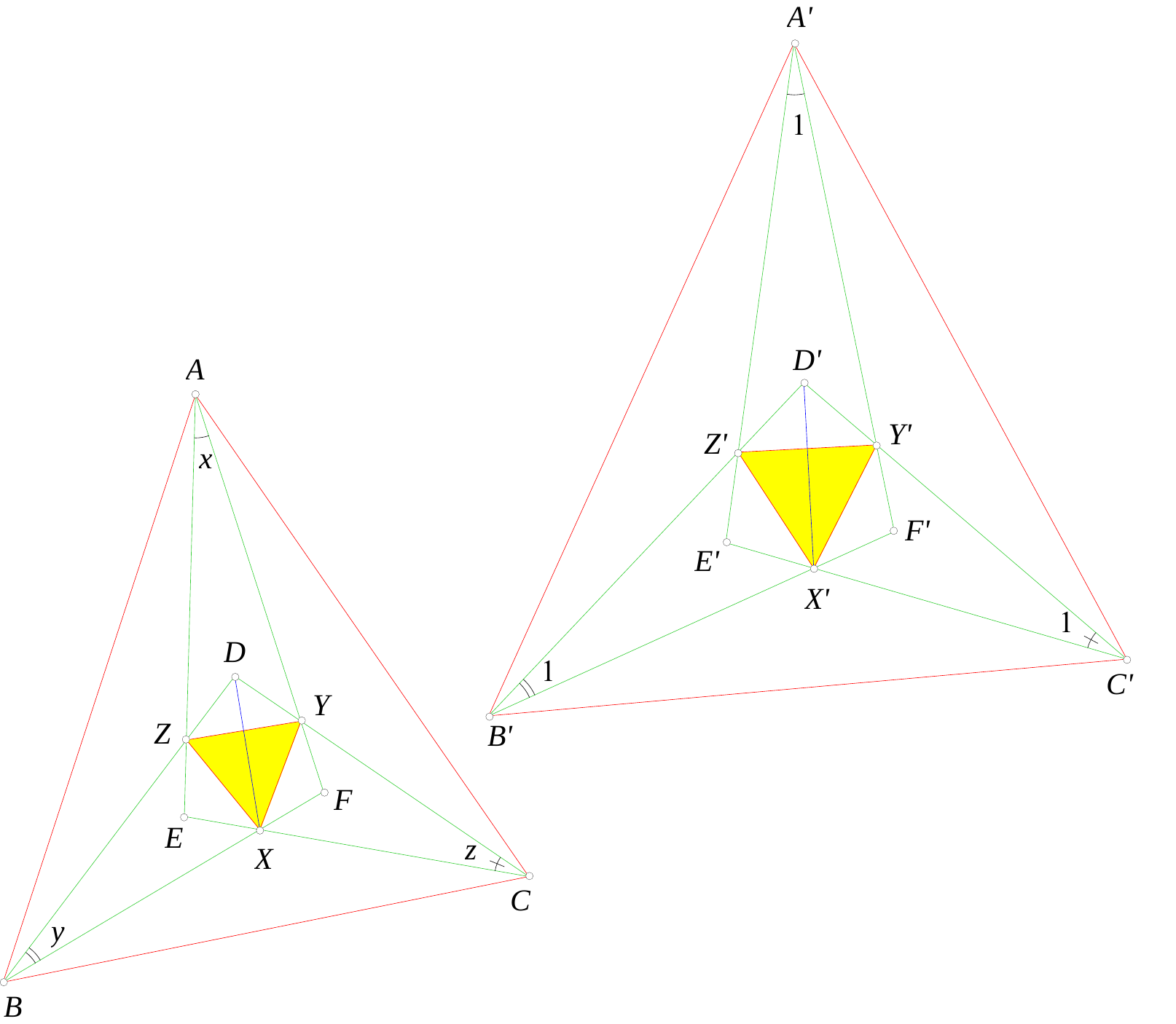}}\end{center}
	\label{fig2}
	\caption{Proof of generalized theorem}
\end{figure}

\begin{proof}[Proof of Theorem \ref{thm2}]The main idea of this proof comes from \cite{2}. (See Figure \ref{fig2}). Set $\angle YAZ = x$, $\angle ZBX = y$, and $\angle XCY = z$. Since $X$ lies inside triangle $DBC$ (because $X$ lies inside triangle $ABC$ and it lies on bisector of angle $\angle BDC$, too), we have
	$$\angle BDC = \angle BXC - y - z = 120^ \circ + x - y - z.$$ 
	
	Similarly, $\angle CEA = 120^ \circ + y - z - x$ and $\angle AFB = 120^ \circ + z - x - y$.
	
	On the sides of an equilateral triangle ${X}'{Y}'{Z}'$, the isosceles triangles ${D}'{Z}'{Y}'$, ${E}'{X}'{Z}'$, and ${F}'{Y}'{X}'$ are constructed outwardly such that $\angle Y'D'Z' = \angle YDZ$, $\angle Z'E'X' = \angle ZEX$, and $\angle X'F'Y' = \angle XFY$.
	
	Take the intersections ${A}' = {E}'{Z}' \cap {F}'{Y}'$, ${B}' = {F}'{X}' \cap D{Z}'$, ${C}' = {D}'{Y}' \cap {E}'{X}'$. From quadrilateral ${A}'{E}'{X}'{F}'$, it is deduced that
	\begin{multline*}\angle {A}'_1 = 360^\circ - \angle Z'E'X' -\\ \left[\left( {90^ \circ - \frac{\angle Z'E'X'}{2}} \right) + 60^\circ + \left(90^ \circ - \frac{\angle X'F'Y'}{2} \right)\right] - \angle X'F'Y',
	\end{multline*}
	this implies that 
	$$\angle {A}'_1 = 120^\circ - \frac{\angle Z'E'X' + \angle X'F'Y'}{2} = 120^\circ - \frac{240^\circ - 2x}{2} = x.$$
	
	An analogous argument shows that
	$$\angle {B}'_1 = y\quad\text{and}\quad\angle {C}'_1 = z.$$
	
	Since ${D}'{X}'$ is bisector of $\angle B'D'C'$ (from the constructions of isosceles triangle $D'Z'Y'$ and equilateral triangle $X'Y'Z'$), $\triangle DBX \sim \triangle {D}'{B}'{X}'$ (because they have same angles $y$, $\frac{\angle BDC}{2}$), and $\triangle DCX \sim \triangle {D}'{C}'{X}'$ (because they have same angles $z$, $\frac{\angle BDC}{2}$), we obtain 
	$$\frac{XB}{{X}'{B}'} = \frac{DX}{{D}'{X}'} = \frac{XC}{{X}'{C}'}\quad\text{or}\quad\frac{XB}{XC} = \frac{{X}'{B}'}{{X}'{C}'},$$
	and also
	$$\angle {B}'{X}'{C}' = \angle B'D'C' + \angle {B}'_1 + \angle {C}'_1 = \angle BXC.$$
	
	Two previous conditions point out that $\triangle XBC \sim \triangle {X}'{B}'{C}'$ (s.a.s). 
	
	Analogously, $\triangle YCA \sim\triangle {Y}'{C}'{A}'$, and $\triangle ZAB 
	\sim \triangle {Z}'{A}'{B}'$. 
	
	Finally, from these similar triangles, it follows that $\angle BAC = \angle B'A'C'$, $\angle CBA = \angle C'B'A'$, and $\angle ACB = \angle A'C'B'$, so $\triangle ABC \sim \triangle {A}'{B}'{C}'$.
	
	This takes us to the conclusion that $\triangle XYZ \sim\triangle {X}'{Y}'{Z}'$, it might be worth pointing out that $XYZ$ is equilateral, and completes the proof of generalized theorem.
\end{proof}

The above proof of generalized theorem also shows that Morley's Theorem can be proven simply using similar triangles and angle chasing in the same way.

The barycentric coordinates will be used in the next proof for Theorem \ref{thm3}, see \cite{5}.

\begin{figure}[htbp]\begin{center}\scalebox{0.7}{\includegraphics{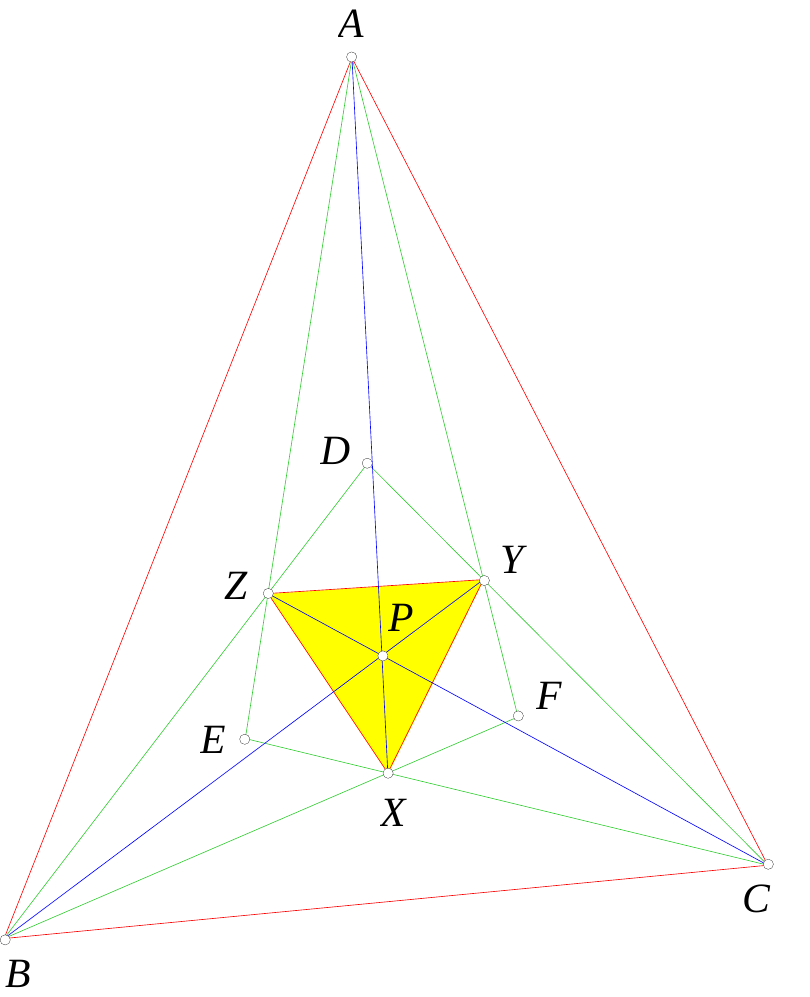}}\end{center}
	\label{fig3}\caption{Perspective triangles}
\end{figure}

\begin{proof}[Proof of theorem \ref{thm3}]Without loss of generality, assume that the sidelengths of the equilateral triangle $XYZ$ is $1$. Therefore, in barycentric coordinates
	$$X = \left( {1:0:0} \right),\ Y = \left( {0:1:0} \right), Z = \left( {0:0:1} \right).$$
	
	Because $D$, $E$, and $F$ lie on perpendicular bisector of sides $BC$, $CA$, and $AB$, respectively, assume that coordinates of $D$, $E$, and $F$ as follows:
	$$D = \left( { - p:1:1} \right),\ E = \left( {1: - q:1} \right),\ F = \left( {1:1: - r} \right).$$ 
	
	Now using the equation of lines \cite{5}, we find that
	$$A = \left( { - 1:q:r} \right),\ B = \left( {p: - 1:r} \right),\ C = \left( {p:q: - 1} \right).$$
	
	Obviously, the lines $AX$, $BY$, and $CZ$ concur at the point $P\left( {p:q:r} \right)$. This finishes the proof.
\end{proof}

\section{Converse construction} In this section, some newly discovered equilateral triangles based on a given arbitrary triangle are found. 

Now coming back to Theorem \ref{thm2}, even though it is really a generalization of Theorem \ref{thm1} and has been proven, we only see one possible case that is Morley's Theorem. That will be make less sense if we only see one application of generalized theorem which is Theorem \ref{thm1}. In order to exclude the objection that it can not find a triangle $XYZ$, satisfying the conditions of the generalized theorem except only the case of Morley's triangle, we now show a converse construction with giving a purely synthetic proof.

\begin{theorem}[Converse construction]\label{thm4}Arbitrary isosceles triangles $DYZ$, $EZX$, and $FXY$ are constructed outwardly of an equilateral triangle $XYZ$ with bases the sides of $XYZ$, such that the pairs of lines $(EZ,FY)$, $(FX,DZ)$, and $(DY,EX)$ meet at $A$, $B$, and $C$ in the same place with $D$, $E$, and $F$, respectively, relative to the sides of $XYZ$. Then 
	$$\angle BXC = 120^ \circ + \angle ZAY,\ \angle CYA = 120^ \circ + \angle XBZ,\ \angle AZB = 120^ \circ + \angle YCX.$$
\end{theorem}
\begin{figure}[htbp]\begin{center}\scalebox{0.7}{\includegraphics{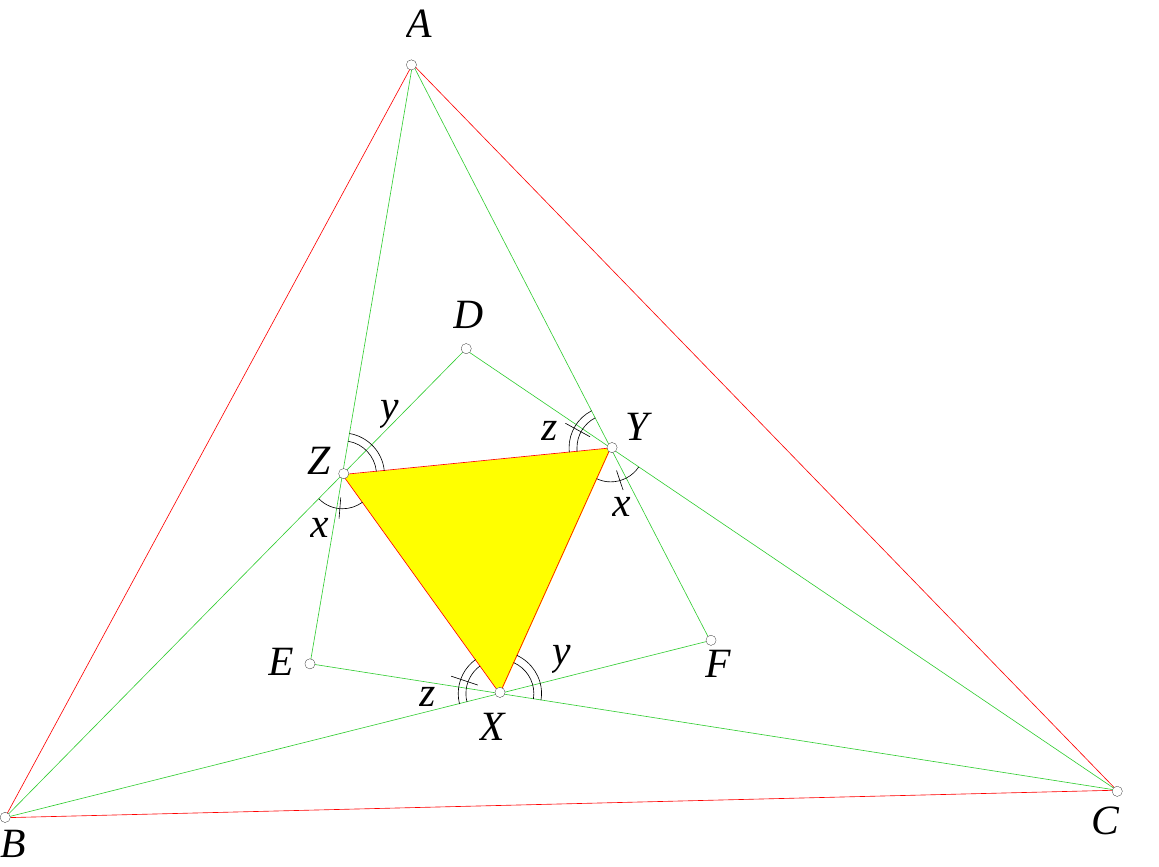}}\end{center}
	\label{fig4}
	\caption{Proof of converse construction}
\end{figure}
\begin{proof}Since $X$ lies on the bisector of $\angle D$, and the sides $XY$, $XZ$ of the equilateral triangle $XYZ$ are equally inclined to the sides $DB$, $DC$ and so we have the equality of angles designated as $x$. Analogously, we have the equality of angles designated $y$ and $z$. So 
	$$\angle BXC = 360^ \circ - y - 60^ \circ - z = 120^ \circ + (180^\circ - y - z) = 120^ \circ + \angle ZAY.$$
	
	Similarly, $\angle CYA = 120^ \circ + \angle XBZ$, $\angle AZB = 120^ \circ + \angle YCX$, this would finish the proof.
\end{proof}

\begin{figure}[htbp]\begin{center}\scalebox{0.7}{\includegraphics{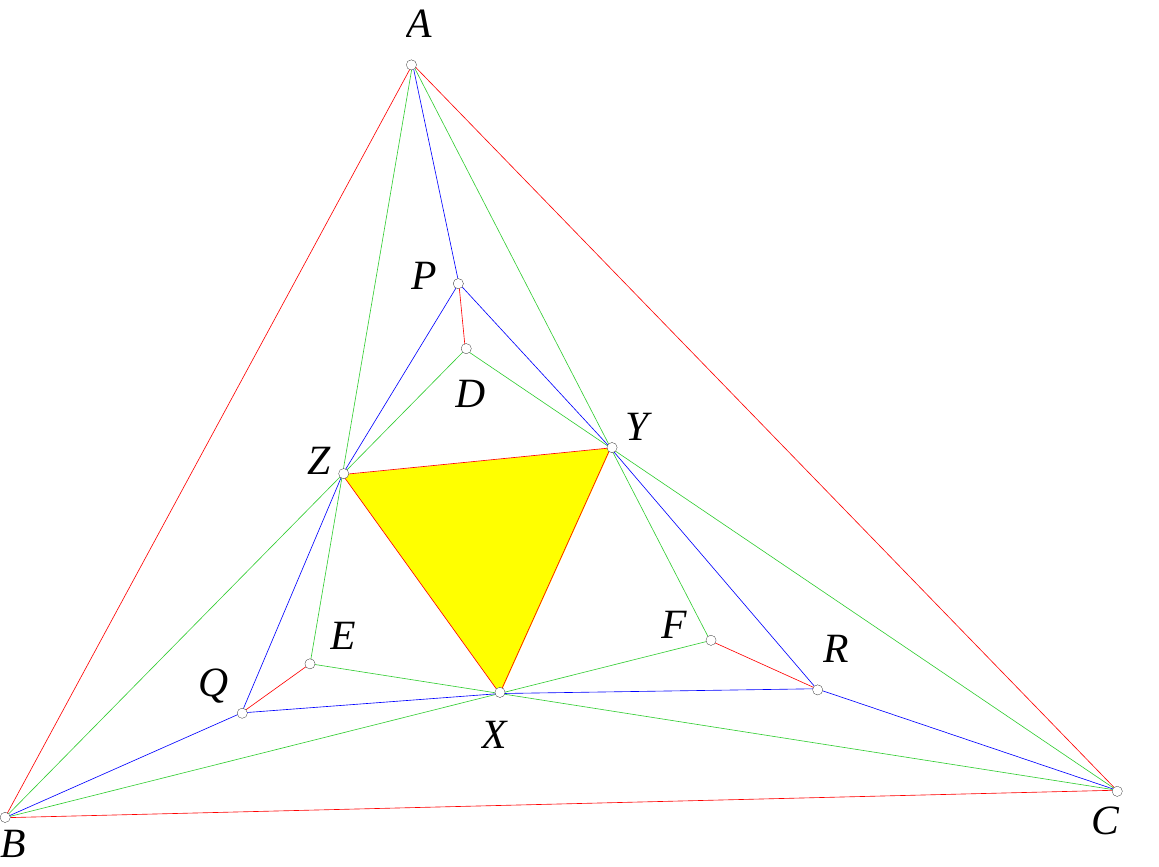}}\end{center}
	\label{fig5}
	\caption{On the configuration of Theorem \ref{thm4}}
\end{figure}
On the configuration of Theorem \ref{thm4}, let $P$, $Q$, and $R$ be the circumcenters of triangles $AYZ$, $BZX$, and $CXY$, respectively. We use angle chasing
$$\angle CXR=\angle BXQ=90^\circ-x,$$
so
$$\angle BXR=\angle QXC=360^\circ-y-z-60^\circ+90^\circ-x=390^\circ-x-y-z.$$

This means that there are six equal angles
$$\angle BXR=\angle CXQ=\angle CYP=\angle AYR=\angle AZQ=\angle BZP.$$

At this point, using above conditions of angles as hypothesis, we propose another extension of Morley's Theorem as follows:

\begin{theorem}[Extension of Morley's Theorem]\label{thm5}Locate the points $X$, $Y$, and $Z$ lying inside a given triangle $ABC$ such that 
	$$\angle BXR=\angle CXQ=\angle CYP=\angle AYR=\angle AZQ=\angle BZP,$$
	where $P$, $Q$, and $R$ are circumcenters of triangles $AYZ$, $BZX$, and $CXY$, respectively, and lying inside that respective triangles. Then triangle $XYZ$ is an equilateral triangle.
\end{theorem}

\begin{figure}[htbp]\begin{center}\scalebox{0.7}{\includegraphics{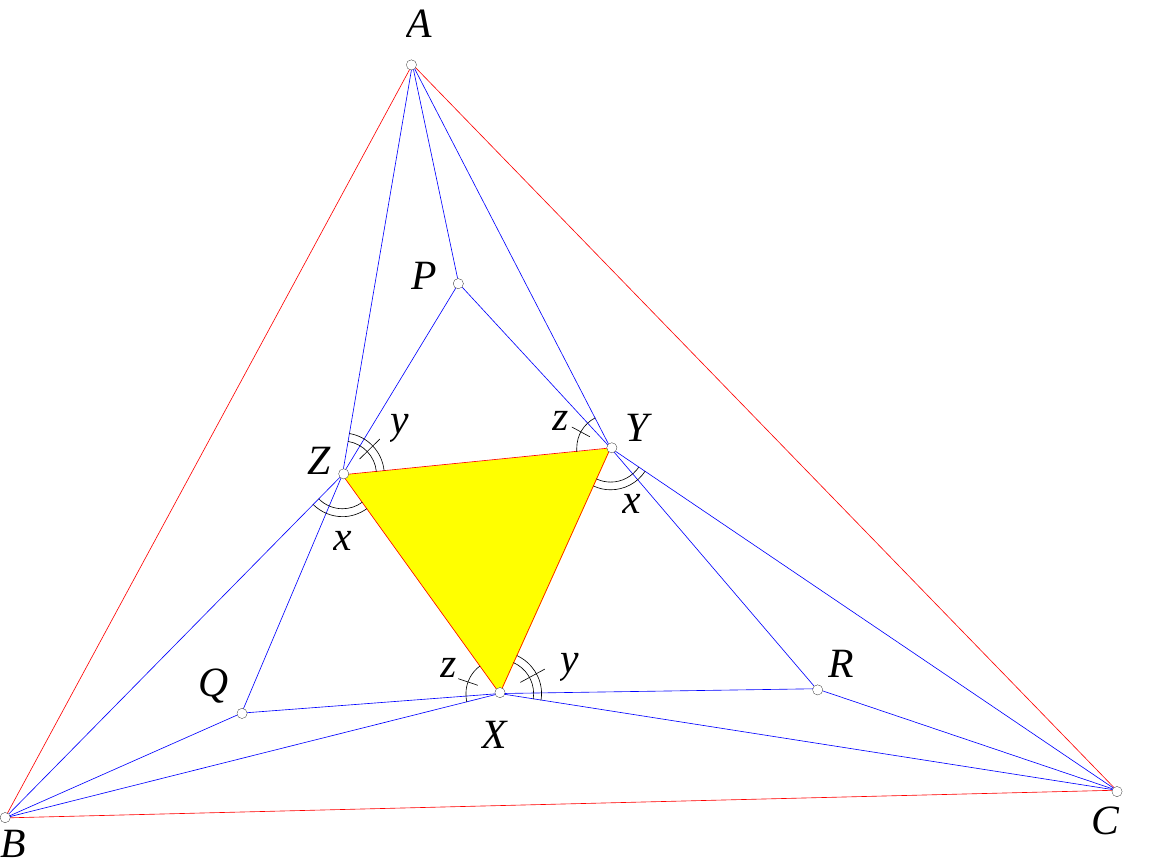}}\end{center}
	\label{fig6}
	\caption{Extension of Morley's Theorem}
\end{figure}

\begin{proof}From the hypothesis we conclude that $\angle CXR = \angle QXB$, this leads to
	$$90^ \circ - \angle XYC = 90^ \circ - \angle BZX,$$ 
	so
	$$\angle XYC = \angle BZX = x.$$
	
	Similarly, we conclude the designation of angles $y$ and angles $z$. From these
	$$\angle BXR = \angle CXX = 360^ \circ - y - z - \angle X + 90^ \circ - x = 450^ \circ- x - y - z - \angle X.$$
	
	Hence, we get the same equalities
	$$\angle CYP=\angle AYR= 450^ \circ - x - y - z - \angle Y,$$
	and
	$$\angle AZQ=\angle BZP= 450^\circ - x - y - z - \angle Z.$$
	
	Thus from six equal angles of the hypothesis, it is easy to show that 
	$$\angle X = \angle Y = \angle Z.$$ 
	
	Therefore, the triangle $XYZ$ is equilateral. The theorem is proven.
\end{proof}

Note that Theorem \ref{thm5} will become Morley's Theorem if adding more the conditions
$$\angle BXR=\angle CXQ=\angle CYP=\angle AYR=\angle AZQ=\angle BZP=150^\circ.$$

Finally, we conclude the article with an interesting consequence of Theorem \ref{thm5} where all six equal angles (in Theorem \ref{thm5}) are $180^\circ$ (see Figure 7).
\begin{figure}[htbp]\begin{center}\scalebox{0.7}{\includegraphics{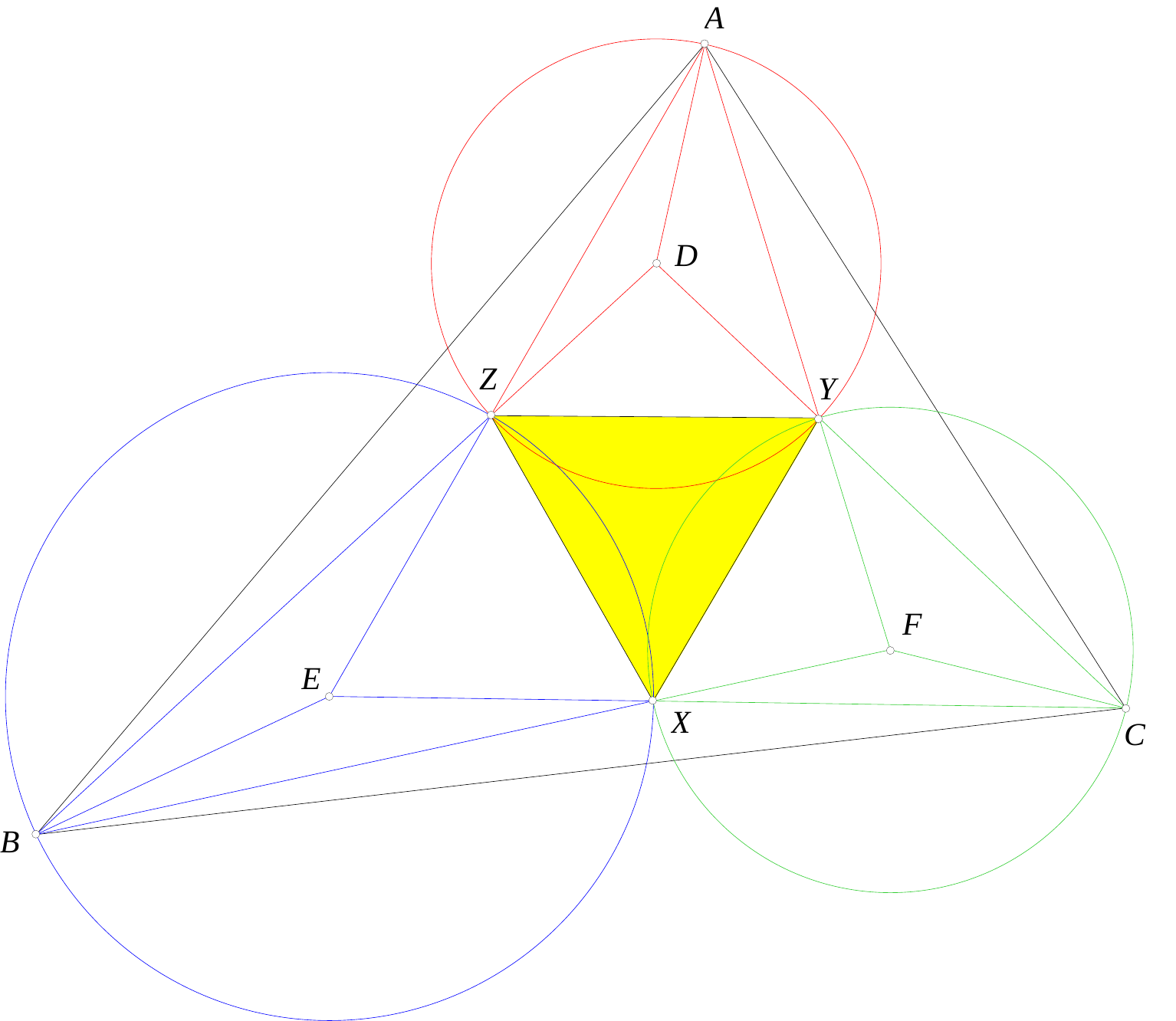}}\end{center}
	\label{fig7}
	\caption{Consequence of Theorem \ref{thm5}}
\end{figure}
\begin{theorem}[Consequence of Theorem \ref{thm5}]\label{thm6}Select three points $X$, $Y$, and $Z$ lying inside a given triangle $ABC$ and satisfying the following conditions
	\begin{itemize}	
		\item $BZ$ and $CY$ meet at circumcenter of triangle $AYZ$. 
		\item $CX$ and $AZ$ meet at circumcenter of triangle $BZX$. 
		\item $AY$ and $BX$ meet at circumcenter of triangle $CXY$. 
	\end{itemize}
Then triangle $XYZ$ is an equilateral triangle. 
\end{theorem}

\begin{acknowledgment}The authors would like to express their sincere gratitude and devote the most respect to two deceased mathematicians Alexander Bogomolny and Richard Kenneth Guy who devoted their love and appreciation to the recreational mathematics, and they have also made a great contribution to the development process and introducing the famous theorem of Morley.
\end{acknowledgment}

\end{document}